\documentclass[10pt]{article}
\usepackage{amsmath, amsfonts, amsthm, amssymb, hyperref, setspace, geometry, authblk, color}
\title{\textbf{A note on the codegree of finite groups}}
\author[1]{Mark L. Lewis}
\author[1]{Quanfu Yan\footnote{Corresponding author.}}

\affil[1]{Department of Mathematical Sciences, Kent State University, Kent, OH 44242, USA}
{
    \makeatletter
    \renewcommand\AB@affilsepx{: \protect\Affilfont}
    \makeatother

    \affil[ ]{\par Email addresses}
    
    \makeatletter
    \renewcommand\AB@affilsepx{, \protect\Affilfont}
    \makeatother

   \affil[ ]{\href{mailto:lewis@math.kent.edu}{lewis@math.kent.edu}}
    \affil[ ]{\href{mailto:qyan5@kent.edu}{qyan5@kent.edu}}

}

\renewcommand{\Affilfont}{\small\it}
\date{}
\usepackage{abstract}

\newtheorem{theorem}{Theorem}[section]
\newtheorem{proposition}[theorem]{Proposition}
\newtheorem{lemma}[theorem]{Lemma}

\theoremstyle{definition}

\newtheorem{example}[theorem]{Example}
\newtheorem*{question*}{Question A}
\allowdisplaybreaks

\expandafter\let\expandafter\oldproof\csname\string\proof\endcsname
\let\oldendproof\endproof
\renewenvironment{proof}[1][\proofname]{%
  \oldproof[\bfseries\scshape #1]%
}{\oldendproof}

\usepackage{stackengine,scalerel}
\stackMath

\renewcommand{\leq}{\leqslant}
\renewcommand{\geq}{\geqslant}
\begin{document}
\maketitle

\begin{abstract}
\noindent\textbf{Abstract.} Let $\chi$ be an irreducible character of a group $G,$ and $S_c(G)=\sum_{\chi\in {\rm Irr}(G)}{\rm cod}(\chi)$ be the sum of the codegrees of the irreducible characters of $G.$ Write ${\rm fcod} (G)=\frac{S_c(G)}{|G|}.$ We aim to explore the structure of finite groups in terms of ${\rm fcod} (G).$ On the other hand, we determine the lower bound of $S_c(G)$ for nonsolvable groups and prove that if $G$ is nonsolvable, then $S_c(G)\geq S_c(A_5)=68,$ with equality if and only if $G\cong A_5.$ Additionally, we show that there is a solvable group so that it has the codegree sum as $A_5.$

\medskip


\end{abstract}

\section{Introduction}

All groups considered in this note are finite. Let $G$ be a group and let ${\rm Irr}(G)$ be the set of complex irreducible characters of $G$. For
$\chi\in {\rm Irr}(G),$ define the number $${\rm cod}(\chi)=\frac{|G:\ker\chi|}{\chi(1)}$$ to be the {\textbf{codegree}} of $\chi.$ This definition of codegree was first introduced by Qian, Wang and Wei in \cite{Qian1}. There is a very large number of recent articles on character coddegrees, for example \cite{Isaacs2011,Qian2021,Yang2017,Du2016}. 

Many authors often explore group theory invariants related to character degrees and study how these invariants may influence the structure of $G$. The average character degree (denoted by ${\rm acd}(G)$) is such an invariant. For example, it was proved in \cite{ILM2013} by Isaacs, Loukaki and Moreto that if ${\rm acd}(G)\leq 3$, ${\rm acd}(G)<\frac{3}{2}$ or ${\rm acd}(G)< \frac{4}{3},$ then $G$ is respectively solvable, supersolvable or nilpotent. Another closely related invariant is $f(G)=\frac{\sum_{\chi\in {\rm Irr}(G)}\chi(1)}{|G|}.$ This invariant has been investigated in many papers, see for examples \cite[Chapter 11]{BerZh1997} and \cite{Moreto2015,PanLi2017,PDY2021}. Recently, Pan, Dong and Yang in \cite{PDY2021} proved that if $f(G)>2/3$ or $f(G)>2$, then $G$ is nilpotent or supersolvable respectively. Motivated by these papers, it is natural to explore the character codegree versions of known results for character degrees.

In this paper, we denote the sum of codegrees of irreducible characters of $G$ by $$S_{c}(G)=\sum_{\chi\in {\rm Irr}(G)}{\rm cod}(\chi),$$ and let ${\rm acod}(G)$ denote the average codegree of the irreducible characters of $G.$  Indeed, Wang, Qian, Lv and Chen in \cite{WQLC2024} investigated the invariant ${\rm acod}(G).$ They provided the lower bounds for ${\rm acod}(G)$ in non-solvable and non-supersolvable groups.

We consider the analog of $f(G)$ with the focus on the invariant ${\rm fcod} (G)=\frac{S_c(G)}{|G|}.$ In particular, we see what can be said regarding this invariant. Unfortunately, the examples constructed in Section 2 led us to believe that no similar results can be obtained. We show that there exist families of non-solvable groups $\{G\}$ and $\{H\}$ such that ${\rm fcod}(G)\rightarrow 0$ as $|G|\rightarrow \infty,$ and ${\rm fcod}(H)\rightarrow \infty$ as $|H|\rightarrow \infty.$ This tells us that the invariant ${\rm fcod} (G)$ has so many possibilities. Thus, additionally we discuss the special case when ${\rm fcod} (G)=1$ and provide some observations.

On the other hand, we consider the codegree sum $S_{c}(G)$ and study how it affects the structure of $G.$ Motivated by the results in \cite{AAIsaacs}, we considered the question in \cite{LYan2023}: Let $G$ be a group of order $n$ and $C_n$ a cyclic group of order $n,$ is it true that $S_{c}(G)\leq S_c(C_n)?$ We showed that this inequality holds for many classes of groups such as nilpotent groups, and it remains valid for any finite group whose order is divisible by up to 99 primes. However, it should be noted that the assertion does not hold true in all cases. In this note, we will consider the lower bound of $S_{c}(G)$ and prove  the following:

\begin{theorem}\label{thm1.1}
Let $G$ be a non-solvable group. Then   $S_c(G)\geq S_c(A_5)=68,$ with equality if and only if $G\cong A_5.$ 
\end{theorem}

From the above theorem, if $G$ is a non-solvable group such that $S_c(G)=S_c(A_5),$  then $G\cong A_5.$ Since the alternating group $A_5$ is the smallest nonabelian simple group, one may ask if it is still true when $G$ is solvable. However, the conclusion is not always true because we noticed that the solvable group $G=A_4\rtimes C_4$ (SmallGroup(48,30)) has the same codegree sum as $A_5.$

The work in this paper was completed by the second author (P.h.D student) under the supervision of the first author at Kent State University. The contents of this paper may appear as part of the second author's P.h.D dissertation.

\section{The invariant ${\rm fcod}(G)$}

We begin by proving the following lemma, which plays an important role in constructing examples. 

\begin{lemma}\label{lem2.1}
Let $G=A\times B$ be a direct product of $A$ and $B$, then $S_{c}(G)\leq S_{c}(A)\cdot S_{c}(B)$ and hence ${\rm fcod}(G)\leq {\rm fcod}(A)\cdot {\rm fcod}(B).$ Furthermore, the equality holds if $(|A|, |B|)=1$. 
\end{lemma}

\begin{proof}
Let $\chi\in {\rm Irr}(A)$ and $\psi\in {\rm Irr}(B)$, then 
$\ker\chi\times \ker\psi\leq \ker(\chi\times \psi)$ and so ${\rm cod}(\chi\times \psi)\leq {\rm cod}(\chi)\cdot {\rm cod}(\psi).$ It follows from {{\cite[Theorem 4.21]{Isaacs1976}}} that the characters $\chi\times \psi$ for $\chi\in {\rm Irr}(A)$ and $\psi\in {\rm Irr}(B)$ are exactly the irreducible characters of $G$. Hence, $S_{c}(G)\leq S_{c}(A)\cdot S_{c}(B)$ and so ${\rm fcod}(G)\leq {\rm fcod}(A)\cdot {\rm fcod}(B).$ 

Now assume  $(|A|, |B|)=1$. Let $g=hk\in \ker(\chi\times \psi)$ for some $h\in A$ and $k\in B.$ Then
$$(\chi\times \psi)(hk)=\chi(h)\psi(k)=\chi(1)\psi(1).$$
It follows that $h\in {Z}(\chi)$ and $k\in {Z}(\psi)$.
Let $\chi_{{Z}(\chi)}=\chi(1)\lambda$ and $\psi_{ Z(\psi)}=\psi(1)\mu$, where $\lambda$ and $\mu$ are linear characters of $ Z(\chi)$ and $Z(\psi)$, respectively.
So $\lambda(h)=\mu(k)^{-1}$. Let $o(g)$ denote the order of $g\in G.$ If either $\lambda(h)\neq 1$ or $\mu(k)\neq 1$, then $o(h)=o(k)$, contrary to $(|H|, |K|)=1$. It follows that $\lambda(h)=\mu(k)=1$ and therefore $h\in \ker{\chi}$ and $k\in \ker\psi$. Hence, $\ker\chi\times \ker\psi=\ker(\chi\times \psi)$, which implies that $${\rm cod}(\chi\times \psi)=\frac{|A\times B: \ker(\chi\times \psi)|}{\chi\times \psi(1)}
={\rm cod}(\chi)\cdot {\rm cod}(\psi).$$
Thus $S_{c}(G)=S_{c}(A)\cdot S_{c}(B)$ and ${\rm fcod}(G)={\rm fcod}(A)\cdot {\rm fcod}(B)$, as wanted.
\end{proof}

The hypothesis $(|A|, |B|)=1$ cannot be dropped, as shown by the examples: $S_c(C_2\times C_2)=7$ but $S_c(C_2)\cdot S_c(C_2)=9.$ Next we compute the values of ${\rm fcod}(G)$ for some specific groups.

\begin{example}\label{examp2.2} (1) If $G$ is a dihedral groups $D_{2^n}$ of order $2^n$ with $n\geq 3,$ then $S_c(G)=1+2\cdot 3+\sum_{i=3}^{n}2^{i-1}\cdot 2^{i-3}=7+\frac43(2^{2n-4}-1)$ and so ${\rm fcod}(G)\rightarrow\infty$ as $n\rightarrow \infty.$

(2) Let $G\cong C_8\rtimes C_2^2$ (SmallGroup(32,43)), then $S_c(G)=31.$ Therefore, ${\rm fcod}(G)=\frac{31}{32}<1.$

(3) ${\rm fcod}(A_5)=\frac{68}{60}>1$ and ${\rm fcod}(A_6)=\frac{311}{360}<1,$ where $A_n$ is the alternating group of degree $n.$ 
\end{example}

Using the groups described above, we are able to construct some families of groups as follows.

\begin{proposition} {\rm (1)} There are families of solvable groups $\{G\}$ and $\{H\}$ such that ${\rm fcod}(G)\rightarrow 0$ as $|G|\rightarrow \infty,$ and ${\rm fcod}(H)\rightarrow \infty$ as $|H|\rightarrow \infty.$ 

{\rm (2)} There exist families of non-solvable groups $\{G\}$ and $\{H\}$ such that ${\rm fcod}(G)\rightarrow 0$ as $|G|\rightarrow \infty,$ and ${\rm fcod}(H)\rightarrow \infty$ as $|H|\rightarrow \infty.$ 
\end{proposition}
\begin{proof} Let $S$ be a group described in Example \ref{examp2.2}(2) and $G=S\times S\times \cdot\cdot\cdot \times S$ be the direct product of $n$ copies of $S.$ It follows from Lemma \ref{lem2.1} that ${\rm fcod}(G)\leq ({\rm fcod}(S))^n=(\frac{31}{32})^n \rightarrow 0$ as $n\rightarrow \infty.$ Together with Example \ref{examp2.2}(1), the statement (1) follows.

Now we prove (2). Let $G$ be the direct product of $n$ copies of $A_6.$ Using a similar proof as above, we have that ${\rm fcod}(G)\rightarrow 0$ as $|G|\rightarrow \infty.$ Now let $H=A_5\times C_p$, where $p>5$ is a prime. Clearly, $H$ is nonsolvable. Then by Lemma \ref{lem2.1}, ${\rm fcod}(H)={\rm fcod}(A_5)\cdot {\rm fcod}(C_p)=\frac{68}{60}\cdot \frac{1+p(p-1)}{p} \rightarrow \infty$ as $p\rightarrow \infty.$ 
\end{proof}

From the above examples and proposition, we may see that  the invariant ${\rm fcod} (G)$ has so many possibilities. Next we will focus on the special case when ${\rm fcod} (G)=1.$ It is clear that the symmetric group $S_3$ is the smallest nontrivial group such that ${\rm fcod} (G)=1.$ It is natural to ask: Are there infinitely many groups satisfying ${\rm fcod}(G)=1?$ How can we describe the stucture of such groups? 

We believe that the first question will have a positive answer based on the following proposition. However, we cannot answer this question completely because of the difficulties in the aspect the number theory.

\begin{proposition} Let $G=C_q \rtimes C_p$ be a non-abelian group of order $pq,$ where $p,q$ are distinct primes. Then ${\rm fcod} (G)=1$ if and only if $q=p^2-p+1.$     
\end{proposition}

\begin{proof} It is obvious that $S_c(G)=1+(p-1)p+\frac{(q-1)q}{p}.$ Then ${\rm fcod} (G)=1$ if and only if $1+(p-1)p+\frac{(q-1)q}{p}=pq,$ which is equivalent to $q^2+(p^2-1)q+p(p^2-p+1)=(q-p)(q-(p^2-p+1))=0.$  The result follows from  $p\not= q$.
\end{proof}

We do find some pairs of $(p,q),$ such as $(2,3), (3,7), (13,157),$ so that $q=p^2-p+1.$ However we do not know if there are infinitely many such prime pairs $(p,q).$ On the other hand, we observe that if $G$ is a $p$-group, then ${\rm fcod}(G)\not=1$.

\begin{proposition} Let $G$ be $p$-group. Then $S_c(G)\equiv 1({\rm mod}\ {p})$ and so ${\rm fcod} (G)\not=1$.
\end{proposition}
\begin{proof} If $\chi\not=1_G,$ we have $|G:{\rm ker}\chi|\geq 1+\chi(1)^2>\chi(1)^2$ and so $1\leq \chi(1)<|G:{\rm ker}\chi|/\chi(1).$ Since $G$ is a $p$-group, $\chi(1)^2$ divides $|G:{\rm ker}\chi|$  for all $\chi\in{\rm Irr}(G).$ It follows that $p|{\rm cod}(\chi)$ if $\chi\not=1_G.$ Hence, $S_c(G)\equiv 1({\rm mod}\ p)$ and so ${\rm fcod} (G)\not=1,$ as wanted.  
\end{proof}

\section{Lower bound of $S_c(G)$ for nonsolvable groups}

In this section, we start by presenting some basic properties of the codegrees, which will be used in our proof. By \cite[Lemma 2.1]{Qian1} and \cite[Lemma 2.3]{Qian1}, we have

\begin{lemma}\label{lem3.1} Let $\chi$ be an irreducible character of $G.$ 

{\rm (1)} For any normal subgroup $N$ of $G$ with $N\leq {\rm ker}\chi$, the codegree ${\rm cod}(\chi)$ is independent of the choice of such $N.$ In particular, $S_c(G/N)\leq S_c(G)$ with equality if and only if $N=1.$

{\rm (2)} If $S$ is a subnormal subgroup of $G$ and $\phi$ is an irreducible constituent of $\chi_S,$ then ${\rm cod}\phi\mid {\rm cod}\chi.$

{\rm (3)} If $G$ is a nonabelian simple group, then for any distinct primes $p,q\in\pi(G)$ where $\pi(G)$ denotes the prime divisors of $|G|,$ there exists $\chi\in {\rm Irr}(G)$ so that $pq\mid {\rm cod}\chi.$
\end{lemma}

When working with nonsolvable groups, it is always necessary to understand some properties about finite nonabelian simple groups. Here our proof relies on the fact (see \cite[Theorem 1]{Herzog1968}):  If $G$ is a finite nonabelian simple group with $|\pi(G)|=3$, then $G$ is
isomorphic to one of the following groups: $A_5,$ $A_6,$ $L_2(7),$ $L_2(8),$ $L_2(17),$ $L_3(3),$ $U_3(3)$ or $U_4(2).$

\begin{proof}[Proof of Theorem \ref{thm1.1}] It suffices to prove that $G\cong A_5$ if $S_c(G)\leq S_c(A_5)=68.$ We work by induction on $|G|.$ Let $M$ be a maximal normal subgroup of $G.$ Then $G/M$ is a simple group. If $G/M$ is nonabelian simple, then $S_c(G/M)\leq S_c(G)=68$ and by induction we have that $M=1$ and $G\cong A_5,$ as desired. 

Now assume that $G/M$ is abelian. Then $G/M\cong C_p$ for some prime $p.$ Let $N$ be a minimal normal subgroup of $G$ contined in $M.$ Consider $\overline{G}=G/N.$ If $N$ is abelian, then $\overline{G}$ is nonsolvable. Since $S_c(\overline{G})< S_c(G)\leq 68,$ by induction $\overline{G}\cong A_5$ and so  $S_c(\overline{G})=68,$ a contradiction. It follows that $N$ is nonabelian. Then $N\cong S_1\times S_2\times \cdot\cdot\cdot \times S_k$, where $S_i\cong S_1=S$ are isomorphic nonabelian simple groups.  We claim that $N\cong S\cong A_5.$ Here we consider the set $\pi(S).$ Assume that $|\pi(S)|>3$ and let $2=p_0<p_1<p_2<p_3$ be the first four smallest prime divisors of $|S|.$ Then $p_3\geq 7.$ Applying Lemma \ref{lem3.1}, there are irreducible  characters $\chi_1,\chi_2,\chi_3\in {\rm Irr}(G)$ so that $p_2p_3\mid {\rm cod}\chi_1, p_1p_3\mid {\rm cod}\chi_2$ and $p_1p_2\mid {\rm cod}\chi_3$ respectively. If $\chi_i=\chi_j$ for some $i\not=j,$ then $p_1p_2p_3\mid {\rm cod}\chi_i.$ Notice that $p_1p_2p_3\geq 3\cdot5\cdot7=105>68$. Then $S_c(G)>{\rm cod}\chi_i\geq 105,$ which is a contradiction. Hence $\chi_i\not=\chi_j$ for any $i\not=j\in\{1,2,3\}.$ It follows that $S_c(G)\geq \sum_{i=1}^3{\rm cod}\chi_i\geq 3\cdot5+3\cdot7\cdot35=71>68.$ This contradiction indicates that $|\pi(S)|=3.$ As we mentioned before, all finite nonabelian simple groups whose orders have three prime divisors have been classified. If $S$ is isomotphic to  $A_6,$ $L_2(8),$ $L_2(17),$ $L_3(3),$ $U_3(3)$ or $U_4(2),$ then by \cite{Conw1985} there is an irreducible character of $S$ having codegree greater than 68. It follows by Lemma \ref{lem3.1}(2), $S_c(G)>68.$ Now assume that $S\cong L_2(7).$ Then by \cite{Conw1985} there are irreducible characters $\phi_1,\phi_2$ of $S$ so that ${\rm cod}\phi_1=2^3\cdot7$ and ${\rm cod}\phi_1=2^3\cdot3.$ If $\phi_1=\phi_2,$ then $2^3\cdot3\cdot7\mid {\rm cod}\phi_1$ and so $S_c(G)>68.$ If $\phi_1\not=\phi_2,$ then $S_c(G)>2^3\cdot7+2^3\cdot3=80>68.$ Both cases are impossible. Hence the only possibility is $S\cong A_5.$ A similar proof works to show that $N\cong S,$ thereby indicating the claim. 
Since $N\cong A_5,$ it follows from Lemma \ref{lem3.1}(2) that $G$ has irreducible characters $\chi_1, \chi_2$ such that $20|{\rm cod}\chi_1$ and $15|{\rm cod}\chi_2.$ In particular, by \cite[Theorem 1.1]{MagTong2011}, let $\theta\in {\rm Irr}(N)$ with $\theta(1)=5,$ then there exists $\chi_3\in {\rm Irr}(G)$ such that $(\chi_3)_N=\theta$ and so $12|{\rm cod}\chi_3.$ It is obvious that any two of these three characters cannot be the same. Hence by the assumption $S_c(G)\leq 68,$ we have that ${\rm cod}\chi_3=12$ or $24.$ Notice that $${\rm cod}\chi_3=\frac{|G:N|\cdot|N|}{5\cdot |{\rm ker}\chi_3|}=12\cdot\frac{|G:N|}{|{\rm ker}\chi_3|}.$$ 
If $\frac{|G:N|}{|{\rm ker}\chi_3|}=1,$ then $G=N\times {\rm ker}\chi_3$ since $N\cap {\rm ker}\chi_3={\rm ker}(\chi_3)_N={\rm ker}\theta=1.$ It follows that $G/{\rm ker}\chi_3\cong N\cong A_5.$ Since $|G:N|\geq |G:M|>1$, we have that ${\rm ker}\chi_3>1$ and hence $S_c(G)>S_c(G/{\rm ker}\chi_3)=68,$ a contradiction. Hence we may have that $|G:{\rm ker}\chi_3|=2|N|=120.$  It follows that $G/{\rm ker}\chi_3\cong S_5$ or $A_5\times C_2,$ both of which have codegree sum greater than 68. This completes the proof.  
\end{proof}

\begin{thebibliography}{0}

\bibitem{AAIsaacs}
H. Amiri, S. M. Jafarian Amiri and I. M. Isaacs, Sums of element orders in finite groups, {\it Comm. Algebra} {\bf 37} (2009), 2978-2980.

\bibitem{Conw1985}
J. H. Conway, R. T. Curtis, S. P. Norton, R. A. Parker and R. A. Wilson, Atlas of Finite Groups, Oxford University Press, London, 1985.

\bibitem{Du2016}
N. Du and M. L. Lewis, Codegrees and nilpotence class of $p$-groups, {\it J. Group Theory} {\bf 19} (2016), 561–567.


\bibitem{Herzog1968}
M. Herzog, On finite simple groups of order divisible by three primes only, {\it J. Algebra} {\bf 10} (1968), 383–388.


\bibitem{Isaacs1976}
I. M. Isaacs, Character Theory of Finite Groups, Academic Press, New York, 1976.


\bibitem{Isaacs2011}
I. M. Isaacs, Element orders and character codegrees, {\it Arch. Math.} {\bf 97} (2011), 499–501.

\bibitem{ILM2013}
I. M. Isaacs, M. Loukaki and A. Moretó, The average degree of an irreducible character of a finite group, {\it Israel J. Math.} {\bf 197} (2013), 55–67.

\bibitem{MagTong2011}
K. Magaard and H. P. Tong-Viet, Character degree sums in finite non-solvable groups. {\it J. Group Theory} {\bf 14}  (2011), 53–57.



\bibitem{Qian1} 
G. Qian, Y. Wang and H. Wei, Co-degrees of irreducible characters in finite groups, {\it J. Algebra} {\bf 312} (2007), 946-955.

\bibitem{Qian2021}
G. Qian, Element orders and codegrees, {\it Bull. London Math. Soc.} {\bf 53} (2021), 820–824.

\bibitem{Yang2017}
Y. Yang and G. Qian, The analog of Huppert’s conjecture on character codegrees, {\it J. Algebra} {\bf 478} (2017), 215–219.







\bibitem{BerZh1997}
Y. Berkovich and E. Zhmud, Characters of finite groups, Part 1, {\it Translations of Mathematical Monographs}, {\bf 172} AMS, Providence, RI, 1997.


\bibitem{Moreto2015}
A. Moret\'{o} and H. N. Nguyen, Character degree sums of finite groups, {\it Forum Math.} {\bf 27(4)} (2015), 2453–2465.

\bibitem{PDY2021}
H. Pan, S. Dong and Y. Yang, Two results on the character degree sums, {\it J. Algebra.} {\bf 584} (2021), 243-259.

\bibitem{PanLi2017}
H. F. Pan and  X. H. Li, On the character degree sums, {\it Commun. Algebra.} {\bf 45(3)} (2017) 1211–1217.


\bibitem{WQLC2024}
Z. Wang, G. Qian, H. Lv and G. Chen, On the average codegree of a finite group, {\it J. Algebra its Appl.} (2024), (Online).

\bibitem{LYan2023}
 M. L. Lewis and Q. Yan, On the sum of character codegrees of finite groups, (2023), (Submitted).



\end{thebibliography}
\end{document}